\pdfoutput=1
\documentclass[11pt, a4paper, english]{amsart}
\usepackage{amsmath} 
\usepackage{amsthm} 
\usepackage{amssymb} 
\usepackage[dvipsnames]{xcolor}
\usepackage{amscd} 
\usepackage{mathtools}

\usepackage{subcaption}

\usepackage{array} 
\usepackage{newtxtext,newtxmath} 

\usepackage[cal=boondoxo,scr=euler]{mathalfa}
\usepackage[backref=page,linktocpage]{hyperref} 
\usepackage{cleveref} 
\usepackage{caption} %
\usepackage{graphics,graphicx} 
\usepackage{tikz,tikz-cd} 

\usetikzlibrary{graphs,graphs.standard,calc}

\usetikzlibrary{decorations.pathreplacing,}

\usepackage{enumerate} 

\DeclareMathAlphabet{\mathsf}{OT1}{\sfdefault}{m}{n}

\newcommand{\nocontentsline}[3]{}
\newcommand{\tocless}[2]{\bgroup\let\addcontentsline=\nocontentsline#1{#2}\egroup}

\usepackage[margin=1.25in]{geometry}
\usepackage{verbatim}

\usepackage{scalerel}

\makeatletter
\def\dual#1{\expandafter\dual@aux#1\@nil}
\def\dual@aux#1/#2\@nil{\begin{tabular}{@{}c@{}}#1\\#2\end{tabular}}
\makeatother

\makeatletter
\@namedef{subjclassname@2020}{\textup{2020} Mathematics Subject Classification}
\makeatother

\DeclareMathAlphabet{\amathbb}{U}{bbold}{m}{n}

\hypersetup{
    colorlinks = true,
    linkbordercolor = {white},
    linkcolor = {BrickRed},
    anchorcolor = {black},
    citecolor = {BrickRed},
    filecolor = {cyan},
    menucolor = {BrickRed},
    runcolor = {cyan},
    urlcolor = {black}
}

\usetikzlibrary{automata}
\usepackage{bbm}

\newtheoremstyle{teoremas}
{11pt}
{11pt}
{\itshape}
{}
{\bfseries}
{}
{.5em}
{}

\theoremstyle{teoremas}
\newtheorem{theorem}{Theorem}[section]

\newtheorem{lemma}[theorem]{Lemma}
\newtheorem{proposition}[theorem]{Proposition}

\newtheoremstyle{definition}
{11pt}
{11pt}
{}
{}
{\bfseries}
{}
{.5em}
{}

\theoremstyle{definition}
\newtheorem{definition}[theorem]{Definition}
\newtheorem{thmdefinition}[theorem]{Theorem-Definition}

\newtheorem{question}[theorem]{Question}
\newtheorem{example}[theorem]{Example}
\newtheorem{remark}[theorem]{Remark}

\crefname{theorem}{theorem}{theorems}
\Crefname{theorem}{Theorem}{Theorems}
\crefname{lemma}{lemma}{lemmas}
\Crefname{lemma}{Lemma}{Lemmas}
\crefname{proposition}{proposition}{propositions}
\Crefname{proposition}{Proposition}{Propositions}

\DeclareMathOperator{\rk}{rk}

\newcommand{\M}{\mathsf{M}}
\newcommand{\N}{\mathsf{N}}
\newcommand{\U}{\mathsf{U}}

\newcommand{\LL}{\mathsf{\Lambda}}
\newcommand{\PP}{\mathsf{\Pi}}
\newcommand{\C}{\mathsf{C}}
\renewcommand{\lambda}{\uplambda}

\newcommand{\rank}{\operatorname{rk}}

\renewcommand{\H}{\mathrm{H}}
\newcommand{\CH}{\mathrm{CH}}

\newcommand{\IH}{\mathrm{IH}}

\newcommand{\cL}{\mathcal{L}}
\newcommand{\Rep}{\operatorname{Rep}}
\newcommand{\gr}{\operatorname{gr}}
\newcommand{\VRep}{\operatorname{VRep}}
\newcommand{\Ind}{\operatorname{Ind}}
\newcommand{\Res}{\operatorname{Res}}
\newcommand{\Aut}{\operatorname{Aut}}

\AtBeginDocument{%
   \def\MR#1{}
}

\title[Equivariant Kazhdan--Lusztig deletion formulas for matroids]{Deletion formulas for equivariant Kazhdan--Lusztig polynomials of matroids}

\author[L.~Ferroni]{Luis~Ferroni}

\address{(L.~Ferroni)
  Institute for Advanced Study, Princeton, NJ, USA
}
\email{ferroni@ias.edu}

\thanks{Luis Ferroni is a member at the Institute for Advanced Study, supported by the Minerva Research Foundation, and was also partially supported by the Swedish Research Council, Grant 2018-03968. Jacob Matherne received support from a Simons Foundation Travel Support for Mathematicians Award MPS-TSM00007970.}

\author[J.~P.~Matherne]{Jacob~P.~Matherne}

\address{(J. P. Matherne)
Department of Mathematics, North Carolina State University, Raleigh, NC, USA.
}
\email{jpmather@ncsu.edu}

\author[L.~Vecchi]{Lorenzo~Vecchi}
\address{(L. Vecchi)
  Department of Mathematics, KTH Royal Institute of Technology, Stockholm, Sweden
}

\email{lvecchi@kth.se}


\allowdisplaybreaks

\begin{document}

\begin{abstract}
      We study equivariant Kazhdan--Lusztig (KL) and $Z$-polynomials of matroids. We formulate an equivariant generalization of a result by Braden and Vysogorets that relates the equivariant KL and $Z$-polynomials of a matroid with those of a single-element deletion. We also discuss the failure of equivariant $\gamma$-positivity for the $Z$-polynomial.  As an application of our main result, we obtain a formula for the equivariant KL polynomial of the graphic matroid gotten by gluing two cycles.  Furthermore, we compute the equivariant KL polynomials of all matroids of corank~$2$ via valuations. This provides an application of the machinery of Elias, Miyata, Proudfoot, and Vecchi to corank $2$ matroids, and it extends results of Ferroni and Schr\"oter.
\end{abstract}

\maketitle

\section{Introduction}

\subsection{Overview}

The Kazhdan--Lusztig polynomials of matroids were introduced by Elias, Proudfoot, and Wakefield in \cite{elias-proudfoot-wakefield}. They are prominent objects in the singular Hodge theory for matroids developed by Braden, Huh, Matherne, Proudfoot, and Wang \cite{braden-huh-matherne-proudfoot-wang}. They take their names due to certain features they share with the classical Kazhdan--Lusztig polynomials arising from intervals in the Bruhat order poset for Coxeter groups \cite{kazhdan-lusztig}. 

One of the most fundamental objects in the theory of Kazhdan--Lusztig (KL) polynomials of matroids is the \emph{intersection cohomology module}. We now recall its construction.  Starting from a matroid $\M$, one constructs the \emph{graded M\"obius algebra} $\H(\M)$ and the augmented Chow ring $\CH(\M)$ (both objects defined in \cite{semismall}). The intersection cohomology module $\IH(\M)$ is the unique (up to isomorphism) indecomposable $\H(\M)$-submodule of $\CH(\M)$ containing the degree-zero piece, $\CH^0(\M)\cong \H^0(\M)$. 
A further object to consider is the \emph{stalk} at the empty flat of $\IH(\M)$, denoted by $\IH(\M)_{\varnothing}$. The Kazhdan--Lusztig polynomial of $\M$ is the Hilbert--Poincar\'e series of $\IH(\M)_{\varnothing}$, whereas the Hilbert--Poincar\'e series of $\IH(\M)$ is the so-called \emph{$Z$-polynomial} of $\M$, introduced and studied first by Proudfoot, Xu, and Young in \cite{proudfoot-xu-young}.

Subtle refinements of the Kazhdan--Lusztig and the $Z$-polynomial have been introduced in \cite{gedeon-proudfoot-young-equivariant,proudfoot-xu-young}, and have been extensively studied in the literature, for example in \cite{proudfoot-equivariant1,xie-zhang,proudfoot-equivariant2,gao-xie-yang,karn-proudfoot-nasr-vecchi,gao-li-xie-yang-zhang}. They are equivariant analogues that keep track of the symmetries of the matroid. Let us recall their definitions. Whenever a group $W$ acts on the ground set of $\M$ preserving its collection of bases, we will write $W \curvearrowright \M$ and call this datum an equivariant matroid. In such a case, we will associate to $\M$ two graded $W$-representations $P_{\M}^W(x)$ and $Z_{\M}^W(x)$ satisfying a certain recursion given in Definition~\ref{def:equivariant-kl-z-polynomials}.  (Note that these polynomials are a priori only graded virtual $W$-representations, but \cite[Theorem 1.3]{braden-huh-matherne-proudfoot-wang} asserts that they are graded honest $W$-representations.)
It is sensible to ask how equivariant Kazhdan--Lusztig and $Z$-polynomials relate to their counterparts for single-element deletions of the matroid. In the non-equivariant setting, a striking recursion found by Braden and Vysogorets \cite{braden-vysogorets} establishes such a relation. The main result of the present paper is an equivariant generalization of their formula.

We note that one obstruction, which was already pointed out in \cite[Remark~2.10]{braden-vysogorets}, is that the action $W \curvearrowright \M$ is not an action on the deletion $\M \smallsetminus \{i\}$. Therefore, any reasonable equivariant deletion formula is, at best, with respect to the action of $W_i$, the stabilizer of the element $i$. In order to state the equivariant deletion formula, we first define for every element $i$ of the ground set the following distinguished subset of the lattice of flats $\mathcal{L}(\M)$:
\begin{equation*}
        \mathscr{S}_i := \mathscr{S}_i(\M) =  \left\{ F \in \mathcal{L}(\M) \mid  F\subsetneq E\smallsetminus\{i\} \text{ and } F\cup\{i\}\in \mathcal{L}(\M)\right\}.
    \end{equation*}
Moreover, we define the equivariant $\tau$-invariant of $\M$ as
    \[
    \tau^W(\M):= \begin{cases}
        [x^j]P_\M^W(x) &\text{if $\rk(\M) = 2j+1$,}\\
        0 & \text{otherwise,}
    \end{cases}
    \]
in analogy with \cite[Definition~2.6]{braden-vysogorets}.
\begin{theorem}\label{thm:equivariant-deletion-formula}
    Let $\M$ be a loopless matroid of rank $k$ on $E$, and let $i\in E$ be an element that is not a coloop. Then,
    \begin{align*}
        P_\M^{W_{i}}(x) &= P_{\M \smallsetminus \{i\}}^{W_i}(x) - x \, P_{\M / \{i\}}^{W_i}(x) \\
        &\qquad + \sum_{[F] \in \mathscr{S}_i/W_i} x^{\frac{k -\rk(F)}{2}} \, \Ind_{W_F\cap W_i}^{W_i}\left( \tau^{W_F \cap W_i}(\M/{(F\cup \{i\})}) \boxtimes P_{\M|_F}^{W_F\cap W_i}(x) \right), \text{ and}\\
        Z_\M^{W_i}(x) &= Z_{\M \smallsetminus \{i\}}^{W_i}(x) + \sum_{[F] \in \mathscr{S}_i/W_i} x^{\frac{k-\rk(F)}{2}} \, \Ind_{W_F\cap W_i}^{W_i}\left( \tau^{W_F \cap W_i}(\M/{(F\cup \{i\})}) \boxtimes Z_{\M|_F}^{W_F\cap W_i}(x) \right),
    \end{align*}
    where $W_F$ is the stabilizer of the flat $F$, and $\mathscr{S}_i/W_i$ is the quotient of $\mathscr{S}_i$ by the group action $W_i$.
\end{theorem}

Among the various reasons to search for a theorem like the preceding one, we will emphasize the following two that led us to it. First, as was established in a prequel \cite{ferroni-matherne-stevens-vecchi} to the present paper, an interesting consequence of the non-equivariant version of the above result is that it leads to a proof of the $\gamma$-positivity of the $Z$-polynomial of $\M$. A well-known generalization of $\gamma$-positivity is the notion of \emph{equivariant $\gamma$-positivity} (see \cite[Section~5.2]{athanasiadis-gamma-positivity} and Definition~\ref{def:eqgammapos} below). As we mentioned in \cite[Remark~4.10]{ferroni-matherne-stevens-vecchi}, and will explain in detail here, one may show that the equivariant $Z$-polynomial fails to be equivariantly  $\gamma$-positive.   The second motivation stems from a computational perspective: the problem of calculating these polynomials both equivariantly and non-equivariantly is notoriously difficult. It is well known that the defining recursions, as well as the geometric interpretation via (stalks of) intersection cohomologies, lend themselves very well for theoretical endeavours, but a significant drawback is that they usually do not simplify the task of actually computing these polynomials. In this direction, we give two applications of  Theorem~\ref{thm:equivariant-deletion-formula} in Section~\ref{sec:applications}.  First, we obtain a formula for the equivariant KL polynomial of the graphic matroid gotten by gluing together two cycles along an edge.  As a further application of our main result, we will leverage the framework developed by Ferroni and Schr\"oter \cite{ferroni-schroter} on split matroids, together with the notion of categorical matroid valuative invariant introduced by Elias, Proudfoot, Miyata, and Vecchi \cite{elias-miyata-proudfoot-vecchi}. This will allow us to state helpful formulas that can be used to compute equivariant KL polynomials of arbitrary corank $2$ matroids under certain group actions.

\section{Preliminaries}

Throughout this paper we shall assume that the reader is acquainted with the basic notions of matroid theory, for which a basic reference is \cite{oxley}. We also assume familiarity with the (non-equivariant) Kazhdan--Lusztig theory of matroids \cite{elias-proudfoot-wakefield,gao-xie,proudfoot-xu-young}, and matroid polytopes and valuations, for which we refer to \cite{derksen-fink,ardila-sanchez,stellahedral,ferroni-schroter}.

\subsection{Equivariant Kazhdan--Lusztig and \texorpdfstring{$Z$}{Z}-polynomials} Let us review the notion of equivariant Kazhdan--Lusztig and $Z$-polynomials of matroids, introduced in \cite{gedeon-proudfoot-young-equivariant} and \cite{proudfoot-xu-young}. A further recommended reading is \cite[Appendix~A]{braden-huh-matherne-proudfoot-wang}. The key objects that we study in this paper can be presented via the following result, which serves also as a definition.

\begin{thmdefinition}\label{def:equivariant-kl-z-polynomials}
    There is a unique way of associating to each loopless matroid $\M$ and each group action $W\curvearrowright \M$, two graded virtual representations $P_\M^W(x), Z_\M^W(x) \in \gr\VRep(W)$ in such a way that
    \begin{enumerate}[\normalfont(i)]
        \item If $\rk(\M) = 0$, then $P_\M^W(x) = Z_\M^W(x) = \mathbf{1}_W$.
        \item If $\rk(\M) > 0$, then $\deg P_\M^W(x) < \frac{1}{2}\rk(\M)$.
        \item For every $\M$,
        \[Z_\M^W(x) = \sum_{[F] \in \cL (\M)/W} x^{\rk(F)} \Ind_{W_F}^W P_{\M/F}^{W_F}(x)\]
        is palindromic and has degree $\rk(\M)$.
    \end{enumerate}
    Here, $\mathbf{1}_W$ is the trivial representation\footnote{In previous related work, the trivial representation was denoted by $\tau_W$. We chose a different notation to avoid confusion with the $\tau$-invariant introduced in Theorem~\ref{thm:equivariant-deletion-formula}.} of $W$, $W_F$ is the stabilizer of $F$, and $\mathcal{L}(\M)/W$ denotes the quotient of the lattice of flats of $\M$ by the group action $W$.
\end{thmdefinition}

\begin{remark}
    The graded representations resulting from the above statement may be viewed as a categorification of the ordinary Kazhdan--Lusztig polynomial $P_\M(x)$ and $Z$-polynomial $Z_\M(x)$. In fact, one has that the $j$-th coefficient 
        \[
            [x^j]P_\M(x) = \dim \left([x^j] P_\M^W(x)\right)
        \]
    for every group action $W \curvearrowright \M$, and similarly for $Z_\M^W(x)$. In particular, if one considers $W$ to be a single-element group, one recovers exactly the ordinary polynomials.
\end{remark}

The following theorem completely characterizes the coefficients of the equivariant Kazhdan--Lusztig polynomial for uniform matroids under the action of the symmetric group $\mathfrak{S}_n$ in terms of irreducible representations (which are in bijection with Young diagrams with $n$ boxes). See also \cite[Theorem~3.7]{gao-xie-yang} for a formula involving skew Specht modules. Let $V_\uplambda$ denote the Specht module associated to the Young diagram (equivalently, the partition) of shape $\uplambda$.

\begin{theorem}[{\cite[Theorem~3.1]{gedeon-proudfoot-young-equivariant}}]\label{thm:unif-equiv-sn}
    Let $\U_{k,n}$ be the rank $k$ uniform matroid with $n$ elements. The constant term of $P_{\U_{k,n}}^{\mathfrak{S}_n}(x)$ is the trivial $\mathfrak{S}_n$-representation $V_{[n]}$.
    The $j$-th coefficient is
    \[
    [x^j]P_{\U_{k,n}}^{\mathfrak{S}_n}(x) = \sum_{b=1}^{\min(n-k,k-2j)} V_{[n-2j-b+1,b+1,2^{j-1}]}.
    \]
\end{theorem}

It is generally desirable to produce formulas for our equivariant polynomials with respect to the full automorphism group $\Aut(\M)$ of the matroid or, if that is not achievable, with respect to the largest possible group acting faithfully. In fact, if $W\curvearrowright \M$ is an equivariant matroid where $W \leq \Aut(\M)$ is any group of symmetries acting on the matroid, we can compute either of our equivariant polynomials $F_\M^W(x)$ as the restriction 
\[
F_\M^W(x) = \Res_W^{\Aut(\M)}F_\M^{\Aut(\M)}(x).
\]
In this sense, Theorem~\ref{thm:unif-equiv-sn} provides a complete answer for every possible group action $W\curvearrowright \U_{k,n}$ (where some cumbersome computations may be hidden in performing the restrictions on the coefficients).  In contrast, there is no way in general to recover $P_\M^W(x)$ from the polynomial $P_M^{W_i}(x)$ that we compute using Theorem \ref{thm:equivariant-deletion-formula}.

\subsection{Equivariant gamma-positivity} By definition, the equivariant $Z$-polynomial is palindromic; i.e., for each $0\leq j\leq \rk(\M)$, the coefficients of degree $j$ and $\rk(\M) - j$ are isomorphic as representations of $W$. Moreover, it is proved in \cite[Theorem~1.3]{braden-huh-matherne-proudfoot-wang} that it is also unimodal; i.e., for every $1 \leq j \leq \frac{1}{2}\rk(\M)$, the coefficient $[x^{j-1}]Z_\M^W(x)$ is a direct summand of $[x^j]Z_\M^W(x)$. This suggests a notion of equivariant $\gamma$-positivity, as explained in Athanasiadis' survey \cite[Section~5.2]{athanasiadis-gamma-positivity}.

\begin{definition}\label{def:eqgammapos}
    Every palindromic equivariant polynomial $F^W(x) = \sum_{j=0}^d V_j\,x^j $, where $V_j \in \Rep(W)$, can be rewritten in a unique way as
    \begin{equation}\label{eq:equivariant-gamma}
    F^W(x) = \sum_{j=0}^{\lfloor \frac{d}{2} \rfloor} \Gamma_j \,x^j(1+x)^{d-2j},
    \end{equation}
    where $\Gamma_j$ is a virtual representation of $W$ for every $j$. We say that $F^W(x)$ is \emph{equivariantly $\gamma$-positive}, or just \emph{$\Gamma$-positive}, if each $\Gamma_j \in \Rep(W)$, i.e. each $\Gamma_j$ is an honest (rather than virtual) representation.
\end{definition}

Since it was proved in \cite[Theorem~4.7]{ferroni-matherne-stevens-vecchi} that the non-equivariant $Z$-polynomial is $\gamma$-positive, one could hope that this result could be strengthened by showing that the equivariant $Z$-polynomial is $\Gamma$-positive. However, as is shown by the next example, this does not hold in general.

\begin{example}\label{ex:Gammapositivity-fails}
    Consider $\U_{2,2}$, the Boolean matroid on a ground set with two elements, and the action of its full automorphism group $ W =\Aut(\U_{2,2}) = \mathfrak{S}_2 \curvearrowright \U_{2,2}$. One can compute
    \begin{align*}
        Z_\M^W(x) &= V_{[2]} + (V_{[2]} \oplus V_{[1,1]})x + V_{[2]}x^2 \\
        &= V_{[2]}(1+x)^2 + (V_{[1,1]} \ominus V_{[2]}) x.
    \end{align*}
    Thus, $Z_\M^W(x)$ is not $\Gamma$-positive.
\end{example}

When asking for $\Gamma$-positivity, we should take care in specifying which group is acting. By virtue of the following  result, we can also focus on the largest possible group action for which $Z_\M^W(x)$ is $\Gamma$-positive.

\begin{lemma}\label{lem:restriction}
    If $F^G(x)$ is $\Gamma$-positive, then so is $F^H(x)$ for every subgroup $H\leq G$.
\end{lemma}

\begin{proof}
    Write $F^G(x) = \sum_j \Gamma_j x^j (1+x)^{d-2j}$, where each $\Gamma_j$ is honest. However, note that $F^H(x) = \Res_H^G F^G(x)$ and $\Res_H^G \Gamma_j$ are, by construction, all honest representations as well. 
\end{proof}

\begin{proposition}
    If $W$ has even order, then the Boolean matroid $W \curvearrowright \U_{n,n}$ is not $\Gamma$-positive.
\end{proposition}

\begin{proof}
    Any group $W$ of even order has a subgroup isomorphic to $\mathfrak{S}_2$, hence by Lemma \ref{lem:restriction} it suffices to show that $\Gamma$-positivity fails when $W\cong \mathfrak{S}_2$. To do so, we show that the coefficient $\Gamma_1$ is not an honest representation.
    One can compute that
    \[
        [x^1]Z_{\U_{n,n}}^{\mathfrak{S}_n}(x) = V_{[n]} \oplus V_{[n-1,1]}.
    \]
    By induction (the base case is in Example \ref{ex:Gammapositivity-fails}), we have that
    \[
    \Res_{\mathfrak{S}_2}^{\mathfrak{S}_n} (V_{[n]} \oplus V_{[n-1,1]}) = (n-1) V_{[2]} \oplus V_{[1,1]}.
    \]
    Thus, in the notation of equation~\eqref{eq:equivariant-gamma}, it follows that $\Gamma_1 = V_{[1,1]} \ominus V_{[2]}$, which is not honest.
\end{proof}

This shows that the inductive proof of \cite[Theorem~4.7]{ferroni-matherne-stevens-vecchi} cannot be extended to the equivariant setting. The base cases of that induction were precisely the Boolean matroids, and since these are not $\Gamma$-positive in general, here we lack the base case of the induction for a generic equivariant matroid $W\curvearrowright \M$. However, one could still ask whether the ``induction step'' holds. More precisely, it is reasonable to ask for an equivariant analogue of the deletion formulas of Braden and Vysogorets. As we will see below, such formulas do exist.

Before finishing this section let us comment that, to the best of our knowledge, this is the first instance of a family of combinatorial polynomials arising ``in nature'', failing to be equivariantly $\gamma$-positive, yet having the property of being (non-equivariantly) $\gamma$-positive.  

\section{Equivariant deletion formula}\label{sec:equivariant-deletion-formula}

By carefully modifying and generalizing the strategy employed by Braden and Vysogorets  \cite[Theorem~2.8]{braden-vysogorets}, we  prove the equivariant deletion formula presented in Theorem~\ref{thm:equivariant-deletion-formula}.

\begin{proof}[Proof of Theorem~\ref{thm:equivariant-deletion-formula}]
    Let $\operatorname{\gr}_\mathbb{Z}\VRep(W) = \VRep(W) [x^{\pm 1}]$ denote the ring of $\mathbb{Z}$-graded virtual representations, i.e. the ring of Laurent polynomials over the ring of virtual representations, and let $\mathcal{H}^W(\M)$ be the free module over $\operatorname{\gr}_\mathbb{Z}\VRep(W)$ with basis indexed by $\mathcal{L}(\M)/{W}$. Define  
    \[
    \upzeta[E]^{W} := \sum_{[G] \in \mathcal{L}(\M)/W} \left(\upzeta_G^E\right)^W \, [G],
    \]
    where $\left(\upzeta_G^E\right)^W := x^{\rk_{\M}(E) - \rk_{\M}(G)} \Ind_{W_G}^W\left(P_{\M/G}^{W_G}(x^{-2}) \right)$. For a proper flat $F$, define $\upzeta[F]^{W_F}$ similarly in $\mathcal{H}^{W_F}(\M|_F)$. Lastly, define a morphism by letting
    \begin{align*}
    \Delta^{W_i}\colon \mathcal{H}^{W_i}(\M) & \to \mathcal{H}^{W_i}(\M \smallsetminus \{i\})\\
    [F] & \mapsto x^{\rk_{\M \smallsetminus \{i\}}(F \smallsetminus \{i\}) - \rk_{\M}(F)}[F\smallsetminus \{i\}]
    \end{align*}
    and extending $\gr_\mathbb{Z}\VRep(W_i)$-linearly. Now, the element $\Delta^{W_i}(\upzeta[E]^{W_i}) \in \mathcal{H}^{W_i}(\M\smallsetminus \{i\})$ can be written as
    \begin{align*}
        \Delta^{W_i}(\upzeta[E]^{W_i}) & = \sum_{[F] \in \mathcal{L}(\M)/W_i} \left(\upzeta_F^E \right)^{W_i} \, x^{\rk_{\M \smallsetminus \{i\}}(F \smallsetminus \{i\}) - \rk_{\M}(F)}[F\smallsetminus \{i\}].
    \end{align*}
    Therefore, the coefficient corresponding to $[\varnothing]$ is
    \[
        [\varnothing] \, \Delta^{W_i}(\upzeta[E]^{W_i})  = \left(\upzeta_{\varnothing}^E \right)^{W_i} + x^{-1}\left(\upzeta_{\{i\}}^E \right)^{W_i} = x^k \left( P_\M^{W_i}(x^{-2}) + x^{-2} P_{\M/ \{i\}}^{W_i}(x^{-2}) \right).
    \]
    Similarly, we can write
    \[
    \Delta^{W_i}(\upzeta[E]^{W_i}) = \upzeta[E\smallsetminus \{i\}]^{W_i} + \sum_{\substack{[F] \in \mathscr{S}_i /W_i\\ F \neq E\smallsetminus \{i\}}} \Ind_{W_F\cap W_i}^{W_i}\left( \tau^{W_F\cap W_i}(\M/(F\cup \{i\})) \boxtimes \upzeta[F]^{W_F\cap W_i} \right),
    \]
    and taking again the coefficient of $[\varnothing]$ we obtain
    \begin{multline*}
        [\varnothing]\Delta^{W_i}(\upzeta[E]^{W_i})\\ 
        = x^k \left(P_{\M \smallsetminus \{i\}}^{W_i}(x^{-2}) + \sum_{[F] \in \mathscr{S}_i/W_i} x^{-(k - \rk(F))} \Ind_{W_F\cap W_i}^{W_i} \left( \tau^{W_F\cap W_i}(\M/{(F\cup \{i\})}) \boxtimes P_{\M|_F}^{W_F\cap W_i}(x^{-2})\right)\right).
    \end{multline*}
    Dividing by $x^k$ yields the result for the KL polynomial after a change of variable, with $x$ in place of $x^{-2}$. The proof for $Z_{\M}^{W_i}(x)$ is entirely analogous and relies on the definition of the $\gr_\mathbb{Z}\VRep(W_i)$-module map $\Phi_{\M}^{W_i}\colon\mathcal{H}^{W_i}(\M) \to \gr_\mathbb{Z}\VRep(W_i)$ given by
    \[
     \sum_{[F] \in \mathcal{L}(\M)/W_i} \alpha_F \, [F] \enspace \stackrel{\Phi_{\M}^{W_i}}{\longmapsto}\enspace \sum_{[F] \in \mathcal{L}(\M)/W_i} x^{-\rk_\M(F)} \alpha_F.
    \qedhere\]
\end{proof}

\begin{remark}
    As we commented before, the formula arising from Theorem \ref{thm:equivariant-deletion-formula} provides the correct induction step to show the $\Gamma$-positivity for some equivariant matroids $W\curvearrowright \M$, where we now understand that a positive or negative answer depends on $W$. We stress the fact that the datum of $\M$ alone is not enough. Indeed, for a fixed element $i$, if $Z_{\M \setminus \{i\}}^{W_i}(x)$ is $\Gamma$-positive and simultaneously $Z_{\M|_F}^{W_F \cap W_i}(x)$ is $\Gamma$-positive for every $F \in \mathscr{S}_i$, then $Z_\M^{W_i}$ is also $\Gamma$-positive. The reader could ask if anything can be deduced about $\Gamma$-positivity when the element $i$ we delete is a coloop of $\M$. If $W$ does not fix $i$ the answer is usually negative (see again Example \ref{ex:Gammapositivity-fails}), so we will work with respect to the action of $W_i \cong W_i \times  \{1\} \leq W$, where $W_i$ is a group of symmetries for $\M\setminus \{i\}$ and $\{1\}$ is the trivial group fixing the coloop $i$. In this case,
    \[
    Z_\M^{W_i \times \{1\}}(x) = Z_{\M \setminus \{i\}}^{W_i}(x) \otimes \left(\mathbf{1}_{\{1\}}(1+x) \right).
    \]
    In particular, one can see that if the coefficients $\Gamma_j$ of
    \[
    Z_{\M \setminus \{i\}}^{W_i}(x) = \sum_{j=0}^{\left\lfloor \frac{\rk(\M)-1}{2}\right\rfloor}\Gamma_jx^j(1+x)^{\rk(\M)-2j}
    \]
    are all honest representations, then the coefficients of
    \[
    Z_\M^{W_i \times \{1\}}(x) = \sum_{j=0}^{\left\lfloor \frac{\rk(\M)}{2}\right\rfloor}\left(\Gamma_j \otimes \mathbf{1}_{\{1\}} \right)x^j(1+x)^{\rk(\M)-2j + 1}
    \]
    are all honest representations as well. Concretely, that is why
    \[
    Z_{\U_{2,2}}^{\mathfrak{S}_1 \times \mathfrak{S}_1}(x) = (\mathbf{1}_{\mathfrak{S}_1}(1+x)) \otimes (\mathbf{1}_{\mathfrak{S}_1}(1+x)) = (\mathbf{1}_{\mathfrak{S}_1}\otimes \mathbf{1}_{\mathfrak{S}_1}) (1+x)^2 = \mathbf{1}_{\mathfrak{S}_1 \times \mathfrak{S}_1} (1+x)^2
    \]
    is $\Gamma$-positive, while $Z_{\U_{2,2}}^{\mathfrak{S}_2}(x)$ is not.
\end{remark}

The reader should not be misled to think that equivariant $\gamma$-positivity for $Z$-polynomials \emph{always fails} when the action is not trivial. The following example shows that in some cases this property may indeed hold true.

\begin{example}
    Let $\M = \U_{2,3}$ and $W\cong\mathfrak{S}_2$ be a group acting on $\M$ by permuting a fixed pair of elements in the ground set $E$. A computation using the defining recursion yields
    \[
    Z_{\U_{2,3}}^{\mathfrak{S}_2}(x) = V_{[2]} + (2V_{[2]}\oplus V_{[1^2]})x + V_{[2]}x^2 = V_{[2]}(1+x)^2 + V_{[1^2]}x,
    \]
    which is $\Gamma$-positive. Note again how this result could not be obtained inductively. If we delete the element $i$ that is fixed by the action of $W$, Theorem \ref{thm:equivariant-deletion-formula} tells us that
    \begin{align*}
    Z_{\U_{2,3}}^{\mathfrak{S}_2}(x) &= Z_{\U_{2,2}}^{\mathfrak{S}_2}(x) + x\, \tau^{\mathfrak{S}_2}(\U_{2,3} / \{i\})\\
    &= V_{[2]}(1+x)^2 + (V_{[1^2]} \ominus V_{[2]})x + V_{[2]}x = V_{[2]}(1+x)^2 + V_{[1^2]}x,
    \end{align*}
    where the final result is $\Gamma$-positive even though one of the summands on the right-hand side is not. 
\end{example}

Experiments suggest that $Z_{\U_{n-1,n}}^{\mathfrak{S}_2}(x)$ may be $\Gamma$-positive for every $n\geq 1$, while for larger groups it is not. We have checked this by hand for all $n\leq 15$. This motivates us to pose the following question.

\begin{question}
    Does $\Gamma$-positivity of $Z_{\U_{k,n}}^{\mathfrak{S}_{n-k+1}}(x)$ hold for every $n\geq k \geq 1$? If not, what is the largest group $W$ such that $Z_{\U_{k,n}}^W(x)$ has this property?
\end{question}

\section{Applications: gluing of cycles  and corank 2 matroids}\label{sec:applications}

In this section, we will give two applications of Theorem~\ref{thm:equivariant-deletion-formula}.  The first, in Section~\ref{subsec:glued-cycles}, computes the equivariant KL polynomial of the graphic matroid (with respect to the largest possible group of symmetries)  obtained by gluing together two cycle graphs along an edge.  Using this result, together with valuativity of the equivariant KL polynomials in \cite{elias-miyata-proudfoot-vecchi} (also cf. Section~\ref{subsec:eqvtval}), we obtain in Section~\ref{subsec:corank2} a formula for the equivariant KL polynomials of arbitrary corank $2$ matroids.

\subsection{Equivariant KL polynomial of two glued cycles}\label{subsec:glued-cycles}

The goal of this section is to compute the equivariant KL polynomial of the graphic matroid (with respect to the largest possible group of symmetries)  obtained by gluing together two cycle graphs along an edge.  This is an equivariant version of the computation for parallel connection graphs of \cite[Theorem~3.2]{braden-vysogorets} in the case of gluing cycles. Before giving the actual computation in Section~\ref{sec:gluingcycles}, we devote Section~\ref{sec:repthy} to some general results on the representation theory of the symmetric group.

\subsubsection{Representation theory generalities}\label{sec:repthy}

For any undefined terms or concepts in the representation theory of finite groups (especially as it applies to the representation theory of the symmetric group), we refer to \cite{fulton-harris} or \cite{james}. Given an irreducible $\mathfrak{S}_N$-representation $V_\lambda$, we may restrict to a two-factor Young subgroup $\mathfrak{S}_{d} \times \mathfrak{S}_{N-d} \leq \mathfrak{S}_N$ in the following way:
\begin{equation}\label{eq:c-lambda-mu-nu}
\Res^{\mathfrak{S}_N}_{\mathfrak{S}_{d} \times \mathfrak{S}_{N-d}}V_\lambda = \sum_{\mu,\nu} c(\lambda; \mu,\nu)\; V_\mu \otimes V_\nu,
\end{equation}
where the coefficients $c(\lambda;\mu,\nu)$ are called \emph{Littlewood--Richardson coefficients}, and the $V_\mu \otimes V_\nu$ are irreducible representations of $\mathfrak{S}_{d} \times \mathfrak{S}_{N-d}$ (i.e., $V_\mu$ is an irreducible representation of $\mathfrak{S}_{d}$ and $V_\nu$ is an irreducible representation of $\mathfrak{S}_{N-d}$).

Recall that a semistandard skew Young tableau $T$ of shape $\lambda/\mu$ is a filling of the skew shape $\lambda/\mu$ whose rows are weakly increasing and whose columns are strictly increasing. We say that a filling has content $\nu= [\nu_1,\ldots,\nu_h]$ if $\nu_i$ is the number of times $i$ appears in the filling. Now we read the filling by rows from right to left, and for every $s\geq 1$ we count the number of occurrences of each $i$ in the first $s$ cells. If this is a non-decreasing function of $i$ we call the filling a \emph{Littlewood--Richardson (LR) filling}. The number of LR fillings of $\lambda / \mu$ of content $\nu$ is precisely the quantity $c(\lambda;\mu,\nu)$ of equation~\eqref{eq:c-lambda-mu-nu}.

In view of our application to gluing cycles in Section~\ref{sec:gluingcycles}, we now set $\lambda = [N-2i,2^i]$.  Fortunately, the Littlewood--Richardson coefficients $c(\lambda;\mu,\nu)$ are computable for this choice of $\lambda$, and we compute them in Proposition~\ref{coefficients fat hooks} below. (In general, one should not expect to be able to efficiently decide the value of a given $c(\lambda;\mu,\nu)$.) It is clear that if $c(\lambda;\mu,\nu) \neq 0$, then $\mu = [p, 2^\ell, 1^{d-p-2\ell}]$ and $\nu = [q,2^s, 1^{N-d-q-2s}]$, with $1 \leq p,q \leq N-2i$ and $d-p-\ell, N-d-q-s \leq i$; i.e., in this case the tableaux consist respectively of $d$ and $N-d$ cells and they are contained in the diagram associated to $\lambda$.

\begin{proposition}\label{coefficients fat hooks}
The coefficient $c(\lambda;\mu,\nu)$ is non-zero if and only if one of the following holds:
    \begin{itemize}
        \item $q = N-2i-p$ and $s=p+i-d+\ell$,
        \item $q = N-2i-p+1$ and $s=p+i-d+\ell$,
        \item $q = N-2i-p+1$ and $s=p+i-d+\ell-1$,
        \item $q = N-2i-p+2$ and $s=p+i-d+\ell-1$.
    \end{itemize}
Moreover, in each of these cases $c(\lambda; \mu,\nu)=1$.
\end{proposition}

Note that not every option in the list produces a valid partition for different shapes of $\lambda = [N-2i,2^i]$ and $\mu$. For example, when $p= N-2i$, i.e. $\mu = [N-2i]$, the only valid partition becomes $\nu = [2^i]$. Moreover, if $p$ (resp. $q$) is equal to 1, we assume $\ell$ (resp. $s$) to be equal to zero and if $p$ (resp. $q$) is equal to 2, then $\ell$ (resp. $s$) counts the number of rows with two boxes starting from the second row (e.g. if $\mu=[2^3]$ then $p=2$ and $\ell=2$).

\begin{proof}
   The idea is that we want to count LR fillings  of shape $\lambda / \mu$ with content $\nu$. First we need to place all the $1$'s---note that there are exactly $q$ of them. At least $N-2i - p$ occurrences of $1$ are needed to fill the first row $\lambda_1 / \mu_1$; we can also have at most two more $1$'s to place in the first and second column. Additional $1$'s would violate the condition of being a semistandard Young tableau, and fewer $1$'s would violate the LR filling condition (the first row would end with a number larger than $1$). Next we move on to the $s$ pairs (the ones corresponding to $\nu_j$ for $2 \leq j \leq s+1$). For simplicity we focus on the first case, because the others are similar. The idea here is to fill exactly the first column with numbers from $2$ to $s+1$, no more and no less; at the same time we are also filling the first $s$ cells of the second column. Lastly, we look at the entries $\nu_j$ for $s+2 \leq j \leq s+t+1$; these can only go in one possible way at the end of the second column.
\end{proof}

\begin{example}
We illustrate the strategy of the previous proof when $\lambda = [6,2^7]$ and $\mu = [3,2^2,1^3]$ (see Figure~\ref{fig:3-22-13}). 
For this skew shape $\lambda/\mu$, the possible contents $\nu$ are
\begin{itemize}
    \item $[3,2^2,1^3]$,
    \item $[4,2^2,1^2]$,
    \item $[4,2,1^4]$,
    \item $[5,2,1^3]$.
\end{itemize}
    \begin{figure}
    \begin{tikzpicture}[scale=0.45]
    \draw[step = 1] (0,0) grid (6,1);
    \draw[step = 1] (0,0) grid (2,-7);
    \draw[very thick, red] (0,1) -- (3,1) -- (3,0) -- (2,0) -- (2,-2) -- (1,-2) -- (1,-5) -- (0,-5) -- (0,1);
    \end{tikzpicture}\caption{$\lambda = [6,2^7]$ and $\mu=[3,2^2,1^3]$.}\label{fig:3-22-13}
    \end{figure}
    
    All four cases in Proposition~\ref{coefficients fat hooks} are indeed feasible, and the associated unique fillings are depicted in Figure~\ref{fig:unique-filling}.
\begin{figure}
\begin{tikzpicture}[scale=0.45]
\draw[step = 1] (0,0) grid (6,1);
\draw[step = 1] (0,0) grid (2,-7);
\draw[very thick, red] (0,1) -- (3,1) -- (3,0) -- (2,0) -- (2,-2) -- (1,-2) -- (1,-5) -- (0,-5) -- (0,1);
\node at (3.5,0.5) {$1$};
\node at (4.5,0.5) {$1$};
\node at (5.5,0.5) {$1$};
\node at (1.5,-2.5) {$2$};
\node at (0.5,-5.5) {$2$};
\node at (1.5,-3.5) {$3$};
\node at (0.5,-6.5) {$3$};
\node at (1.5,-4.5) {$4$};
\node at (1.5,-5.5) {$5$};
\node at (1.5,-6.5) {$6$};
\end{tikzpicture}
\qquad
\begin{tikzpicture}[scale=0.45]
\draw[step = 1] (0,0) grid (6,1);
\draw[step = 1] (0,0) grid (2,-7);
\draw[very thick, red] (0,1) -- (3,1) -- (3,0) -- (2,0) -- (2,-2) -- (1,-2) -- (1,-5) -- (0,-5) -- (0,1);
\node at (3.5,0.5) {$1$};
\node at (4.5,0.5) {$1$};
\node at (5.5,0.5) {$1$};
\node at (1.5,-2.5) {$1$};
\node at (0.5,-5.5) {$2$};
\node at (1.5,-3.5) {$2$};
\node at (0.5,-6.5) {$3$};
\node at (1.5,-4.5) {$3$};
\node at (1.5,-5.5) {$4$};
\node at (1.5,-6.5) {$5$};
\end{tikzpicture}
\qquad
\begin{tikzpicture}[scale=0.45]
\draw[step = 1] (0,0) grid (6,1);
\draw[step = 1] (0,0) grid (2,-7);
\draw[very thick, red] (0,1) -- (3,1) -- (3,0) -- (2,0) -- (2,-2) -- (1,-2) -- (1,-5) -- (0,-5) -- (0,1);
\node at (3.5,0.5) {$1$};
\node at (4.5,0.5) {$1$};
\node at (5.5,0.5) {$1$};
\node at (1.5,-2.5) {$2$};
\node at (0.5,-5.5) {$1$};
\node at (1.5,-3.5) {$3$};
\node at (0.5,-6.5) {$2$};
\node at (1.5,-4.5) {$4$};
\node at (1.5,-5.5) {$5$};
\node at (1.5,-6.5) {$6$};
\end{tikzpicture}
\qquad
\begin{tikzpicture}[scale=0.45]
\draw[step = 1] (0,0) grid (6,1);
\draw[step = 1] (0,0) grid (2,-7);
\draw[very thick, red] (0,1) -- (3,1) -- (3,0) -- (2,0) -- (2,-2) -- (1,-2) -- (1,-5) -- (0,-5) -- (0,1);
\node at (3.5,0.5) {$1$};
\node at (4.5,0.5) {$1$};
\node at (5.5,0.5) {$1$};
\node at (1.5,-2.5) {$1$};
\node at (0.5,-5.5) {$1$};
\node at (1.5,-3.5) {$2$};
\node at (0.5,-6.5) {$2$};
\node at (1.5,-4.5) {$3$};
\node at (1.5,-5.5) {$4$};
\node at (1.5,-6.5) {$5$};
\end{tikzpicture}
\caption{The unique filling in each of the four cases}\label{fig:unique-filling}
\end{figure}
\end{example}

\subsubsection{Application to gluing cycles}\label{sec:gluingcycles}

We now apply the results of the previous subsection to compute the equivariant KL polynomial of the graphic matroid associated to the gluing of two cycles along an edge.  This will be also instrumental for the computation of equivariant KL polynomials of corank $2$ matroids. 

First we recall that the graphic matroid $\C_n$ associated to an $n$-cycle graph is isomorphic to the uniform matroid of corank one $\U_{n-1,n}$, so we will use the notation $\C_n$ and $\U_{n-1,n}$ interchangeably. On this matroid, there is a natural action of $\mathfrak{S}_n$ permuting the ground set elements. When we glue two cycles of lengths $a$ and $b$ along a common edge $e$, the resulting graphic matroid, which we will denote as $\C_{a,b}$, inherits an action of $W_e \leq\mathfrak{S}_{a+b-2}$, the stabilizer of $e$. (Notice that this group acts on the matroid, not on the underlying graph.) We a priori do not require that the labels of the cycle of length $a$ and the labels of the cycle of length $b$ are given by intervals of consecutive integers. In particular, for two sets $A$ and $B$ such that $A\cap B=\{e\}$, we will write $\C_A$, $\C_B$, and $\C_{A,B}$ as labelled counterparts of $\C_a$, $\C_b$, and $\C_{a,b}$, where $a = |A|$ and $b=|B|$. As usual, we denote $E = A\cup B = [n]$ for the full ground set of $\C_{A,B}$.

There is an isomorphism $W_e \cong \mathfrak{S}_{A\smallsetminus e} \times \mathfrak{S}_{B\smallsetminus e}$, and thus one may compute via Theorem~\ref{thm:equivariant-deletion-formula} that
\begin{equation}\label{eq:equiv-Cab}
    P_{\C_{A,B}}^{\mathfrak{S}_{A\smallsetminus e}\times \mathfrak{S}_{B\smallsetminus e}}(x) = P_{\C_{E\smallsetminus e}}^{\mathfrak{S}_{A\smallsetminus e}\times \mathfrak{S}_{B\smallsetminus e}}(x) - x P_{\C_{A\smallsetminus e}}^{\mathfrak{S}_{A\smallsetminus e}}(x)\otimes P_{\C_{B\smallsetminus e}}^{\mathfrak{S}_{B\smallsetminus e}}(x).
\end{equation}

To obtain equation~\eqref{eq:equiv-Cab}, we note that, in the case of parallel connection graphs, the sum appearing on the right-hand side of the deletion formula of Theorem~\ref{thm:equivariant-deletion-formula} vanishes.  To see this, we first point out that the analogous non-equivariant computation in \cite{braden-vysogorets} implies that $\dim \tau^{W_F \cap W_i}(\M / (F \cup \{i\})) = 0$, and then \cite[Theorem~1.3 (1)]{braden-huh-matherne-proudfoot-wang}, asserting that equivariant KL polynomials are honest representations, implies that $\tau^{W_F \cap W_i}(\M / (F \cup \{i\})) = 0$ as a representation.

The following useful proposition first appeared in \cite[Theorem~1.2(2)]{proudfoot-wakefield-young} and can be obtained by specializing Theorem~\ref{thm:unif-equiv-sn} to the case of a corank $1$ uniform matroid.

\begin{proposition}\label{prop:kl-corank-1-uniform}
    The $\mathfrak{S}_n$-equivariant Kazhdan--Lusztig polynomial of the corank $1$ uniform matroid $\U_{n-1,n}$ is given by
    \[
    P_{\U_{n-1,n}}^{\mathfrak{S}_n}(x) = P_{\C_n}^{\mathfrak{S}_n}(x) = \sum_{j=0}^{\left\lfloor \frac{n-2}{2} \right\rfloor} V_{[n-2j,2^j]}x^j.
    \]
\end{proposition}

As we will see below, the preceding result and Proposition~\ref{coefficients fat hooks} are the only tools needed for computing the equivariant KL polynomial of two cycles glued together.  In fact, the first term on the right-hand side of equation~\eqref{eq:equiv-Cab} can be rewritten as follows:
    \begin{equation}\label{eq:equiv-corank1-product-sn}
    P_{\C_{E \smallsetminus e}}^{\mathfrak{S}_{A \smallsetminus e}\times \mathfrak{S}_{B\smallsetminus e}}(x) = \Res_{\mathfrak{S}_{A\smallsetminus e}\times \mathfrak{S}_{B\smallsetminus e}}^{\mathfrak{S}_{E \smallsetminus e}} P_{\C_{E \smallsetminus e}}^{\mathfrak{S}_{E \smallsetminus e}}(x).
    \end{equation}
The following example explains how to give a concrete computation for the right-hand side in equation~\eqref{eq:equiv-corank1-product-sn}. 

\begin{example}\label{ex: U89}
    We will compute the equivariant KL polynomial of $(\mathfrak{S}_4\times \mathfrak{S}_5) \curvearrowright \C_{5,6}$, where we abuse notation and write $\C_{5,6}$ as a shorthand for $\C_{A,B}$ for $A=\{1,\ldots,5\}$ and $B=\{5,\ldots,10\}$.

    First, we carry out the computation of equation~\eqref{eq:equiv-corank1-product-sn}.  To this end, consider the equivariant matroid $(\mathfrak{S}_4 \times \mathfrak{S}_5) \curvearrowright \U_{8,9}$ on the groundset $\{1,\ldots,10\}\setminus \{5\}$, where $\mathfrak{S}_5$ is meant to act on the subset of labels $\{6,7,8,9,10\}$. Its equivariant Kazhdan--Lusztig polynomial is given by
    \[
    P_{\U_{8,9}}^{\mathfrak{S}_4 \times \mathfrak{S}_5}(x) = \Res_{\mathfrak{S}_4 \times \mathfrak{S}_5}^{\mathfrak{S}_9} \left( V_{[9]} + V_{[7,2]}x + V_{[5,2^2]}x^2 + V_{[3,2^3]}x^3\right).
    \]
    By carefully substituting all the possible values from Proposition~\ref{coefficients fat hooks} we see that
    \begin{align*}
    [x^0]P_{\U_{8,9}}^{\mathfrak{S}_4 \times \mathfrak{S}_5}(x) &= V_{[4]}\otimes V_{[5]}, \\
    [x^1]P_{\U_{8,9}}^{\mathfrak{S}_4 \times \mathfrak{S}_5}(x) &= (V_{[4]}\otimes V_{[5]}) \oplus (V_{[4]}\otimes V_{[4,1]}) \oplus (V_{[4]}\otimes V_{[3,2]}) \\ 
    &\qquad\oplus (V_{[3,1]}\otimes V_{[5]}) \oplus (V_{[3,1]}\otimes V_{[4,1]}) \oplus (V_{[2^2]}\otimes V_{[5]}),\\
    [x^2]P_{\U_{8,9}}^{\mathfrak{S}_4 \times \mathfrak{S}_5}(x) &= (V_{[4]}\otimes V_{[3,2]})\oplus (V_{[4]}\otimes V_{[2^2,1]})\oplus (V_{[3,1]}\otimes V_{[4,1]}) \\
    &\qquad\oplus (V_{[3,1]}\otimes V_{[3,2]})\oplus (V_{[3,1]}\otimes V_{[3,1^2]})\oplus (V_{[3,1]}\otimes V_{[2^2,1]})\\
    &\qquad\oplus (V_{[2^2]}\otimes V_{[5]})\oplus (V_{[2^2]}\otimes V_{[4,1]})\oplus  (V_{[2^2]}\otimes V_{[3,2]})\\
    &\qquad\oplus (V_{[2,1^2]}\otimes V_{[4,1]})\oplus (V_{[2,1^2]}\otimes V_{[3,1^2]}),\\
    [x^3]P_{\U_{8,9}}^{\mathfrak{S}_4 \times \mathfrak{S}_5}(x) &= (V_{[3,1]}\otimes V_{[2^2,1]})\oplus (V_{[2^2]}\otimes V_{[2^2,1]})\oplus (V_{[2^2]}\otimes V_{[3,2]}) \\
    &\qquad\oplus (V_{[2,1^2]}\otimes V_{[2^2,1]})\oplus (V_{[2,1^2]}\otimes V_{[2,1^3]})\oplus (V_{[2,1^2]}\otimes V_{[3,1^2]}) \\
    &\qquad\oplus (V_{[1^4]}\otimes V_{[2,1^3]}).
    \end{align*}

  Now we move on to the tensor product on the right-hand side of equation~\eqref{eq:equiv-Cab}.   Proposition \ref{prop:kl-corank-1-uniform} asserts that  $P_{\C_4}^{\mathfrak{S}_4}(x) = V_{[4]} + V_{[2^2]}x$ and $P_{\C_5}^{\mathfrak{S}_5}(x) = V_{[5]} + V_{[3,2]}x$, and by subtracting them from the coefficients computed before we obtain
    \begin{align*}
    [x^0]P_{\C_{5,6}}^{\mathfrak{S}_4 \times \mathfrak{S}_5}(x) &= V_{[4]}\otimes V_{[5]}, \\
    [x^1]P_{\C_{5,6}}^{\mathfrak{S}_4 \times \mathfrak{S}_5}(x) &= (V_{[4]}\otimes V_{[4,1]}) \oplus (V_{[4]}\otimes V_{[3,2]})\oplus (V_{[3,1]}\otimes V_{[5]})\\
    & \qquad\oplus (V_{[3,1]}\otimes V_{[4,1]}) \oplus (V_{[2^2]}\otimes V_{[5]}),\\
    [x^2]P_{\C_{5,6}}^{\mathfrak{S}_4 \times \mathfrak{S}_5}(x) &= (V_{[4]}\otimes V_{[2^2,1]})\oplus (V_{[3,1]}\otimes V_{[4,1]}) \oplus (V_{[3,1]}\otimes V_{[3,2]}) \\
    &\qquad\oplus (V_{[3,1]}\otimes V_{[3,1^2]})\oplus (V_{[3,1]}\otimes V_{[2^2,1]})\oplus  (V_{[2^2]}\otimes V_{[4,1]})\\
    &\qquad\oplus  (V_{[2^2]}\otimes V_{[3,2]}) \oplus (V_{[2,1^2]}\otimes V_{[4,1]})\oplus (V_{[2,1^2]}\otimes V_{[3,1^2]}),\\
    [x^3]P_{\C_{5,6}}^{\mathfrak{S}_4 \times \mathfrak{S}_5}(x) &= (V_{[3,1]}\otimes V_{[2^2,1]})\oplus (V_{[2^2]}\otimes V_{[2^2,1]}) \oplus (V_{[2,1^2]}\otimes V_{[2^2,1]})\\
    &\qquad\oplus (V_{[2,1^2]}\otimes V_{[2,1^3]})\oplus (V_{[2,1^2]}\otimes V_{[3,1^2]}) \oplus (V_{[1^4]}\otimes V_{[2,1^3]}).
    \end{align*}
For comparison, the non-equivariant Kazhdan--Lusztig polynomials are listed below:
    \begin{align*}
        P_{\C_{5,6}}(x) &= 74x^3+113x^2+26x+1,\\
        P_{\U_{8,9}}(x) &= 84x^3+120x^2+27x+1.
    \end{align*}
As expected, the non-equivariant coefficients give the dimensions of the corresponding representations.
\end{example}

\subsection{Corank 2 matroids.}

We recall the machinery of Elias, Proudfoot, Miyata, and Vecchi \cite{elias-miyata-proudfoot-vecchi}, which implies  that the equivariant KL polynomials of matroids are matroid valuations. As an application of this fact, Theorem~\ref{thm:equivariant-deletion-formula}, and Section~\ref{subsec:glued-cycles}, we will deduce new formulas for arbitrary matroids of corank $2$, and therefore will generalize results of Ferroni and Schr\"oter \cite[Section~9.1]{ferroni-schroter}.

\subsubsection{Equivariant valuativity for Kazhdan--Lusztig polynomials.} \label{subsec:eqvtval}

We recall the construction from \cite{elias-miyata-proudfoot-vecchi}.  The construction is technical, and we refer the reader to \cite{elias-miyata-proudfoot-vecchi} for details. Consider the category $\mathcal{M}$ of matroids with rank preserving weak maps. For every increasing sequence of non-negative integers $\textbf{k} = (k_1,\ldots, k_r)$ we define the \emph{Whitney functor} $\Phi_{\textbf{k}}\colon \mathcal{M} \to \operatorname{Vec}_\mathbb{Q}$ as follows. For a given $\M$ and such a $\textbf{k}$, we define
    \[
    \mathcal{L}_{\textbf{k}}(\M) = \{F_1\subseteq\cdots \subseteq F_r \mid F_i\in \mathcal{L}(\M) \text { and } \rk(F_i) = k_i \text{ for all $1\leq i\leq r$} \}.
    \]
Then, on objects, $\Phi_{\textbf{k}}(\M)$ is a vector space with basis $\mathcal{L}_\textbf{k}(\M)$ and on morphisms, if $\varphi\colon \M \to \M'$ and $(F_1,\ldots, F_r) \in \mathcal{L}_\textbf{k}(\M)$, then
\[
\Phi_{\textbf{k}}(\varphi)(F_1,\ldots, F_r) = \begin{cases}
    \left(\overline{\varphi(F_1)},\ldots, \overline{\varphi(F_r)}\right) & \text{if $\rank_{\M'}\left(\varphi(F_i)\right) = k_i$ for all $i$,}\\
    0 & \text{otherwise.}
\end{cases}
\]

We now define the \emph{Kazhdan--Lusztig functor} $\operatorname{KL}\colon \mathcal{M} \to \operatorname{bigrVec}_\mathbb{Q}$ into the category of bigraded vector spaces as follows. On matroids with loops we set $\operatorname{KL}(\M) = 0$. On the full subcategory of loopless matroids of rank $k$, we define
\[
\operatorname{KL} := \amathbb{1} \oplus \bigoplus_{r=1}^i \bigoplus_{R \subseteq[r]} \bigoplus_{\substack{a_0 < a_i < \dots < a_r < a_{r+1} \\ a_0 = 0 \\ a_r = i \\ a_{r+1}= k-i}} \Phi_{\textbf{k}}(-i,-|R|),
\]
where $\amathbb{1}$ is the trivial functor that maps every matroid $\M$ to a 1-dimensional $\mathbb{Q}$-vector space, $\textbf{k} = (k_1,\ldots, k_r)$, $k_j = k - a_{s_{r+1 - j}(R)} - a_{r-j}$, and $s_j(R) := \min\{ \ell \in \mathbb{Z}\smallsetminus R \mid \ell\geq j\}$. 
By definition, this functor categorifies
\[
\widetilde{P}_{\M}(t,u) := \sum_{i,j}\dim \operatorname{KL}^{i,j}t^iu^j,
\]
and by \cite[Theorem 6.1]{proudfoot-xu-young}
\[
\widetilde{P}_{\M}(x,-1) := P_\M(x).
\]

Let $\mathcal{N}$ be a matroid subdivision of some matroid polytope $\M$, and let $\Omega$ be an orientation of $\mathcal{N}$. In light of \cite[Corollary~8.15]{elias-miyata-proudfoot-vecchi}, it is possible to write an exact sequence of bigraded vector spaces involving all the internal faces of $\mathcal{N}$. Now let $W$ be a group preserving the subdivision $\mathcal{N}$. For every $\N \in \mathcal{N}$, the bigraded vector space $\operatorname{KL}(\N)$ carries an action of the stabilizer $W_{\N}$. One can show that the virtual graded $W_{\N}$-representation
\[
\sum_j (-1)^j\operatorname{KL}^{i,j}(\N)
\]
is equal to the coefficient $[x^i]P_{\N}^{W_{\N}}(x)$ (see the discussion in \cite[Section~9]{elias-miyata-proudfoot-vecchi}). 

A subset $S$ of a matroid $\M$ is \emph{stressed} if the restriction $\M|_S$ and the contraction $\M/S$ are uniform matroids. A matroid is called \emph{elementary split} if it does not contain minors isomorphic to $\U_{0,1}\oplus \U_{1,2}\oplus \U_{1,1}$ (see \cite[Section~4]{ferroni-schroter}). For these matroids, we obtain the next equation, which follows from \cite[Proposition~9.6]{elias-miyata-proudfoot-vecchi} and provides a categorical generalization of \cite[Theorem~1.4]{ferroni-schroter}.

\begin{theorem}\label{thm:equiv-val-for-split}
    Let $W \curvearrowright \M$ be an equivariant elementary split matroid of rank $k$ on the ground set $E$, and let $\mathcal{F}$ denote its family of stressed flats. Then
    \[
    P_\M^W(x) = P_{\U_{k,n}}^W(x) - \sum_{F \in \mathcal{F}/W} \Ind_{W_F}^W \left(P_{\LL_{r,k,F,E}}^{W_F}(x) - P_{\PP_{r,k,F,E}}^{W_F}(x)\right).
    \]
    The matroids $\LL_{r,k,F,E}$ appearing on the right-hand side are matroids on $E$ having rank $k$ and a unique distinguished stressed flat $E\smallsetminus F$ of rank $k-r$. The matroid $\PP_{r,k,F,E}$ has ground set $E$ and rank $k$, and is given by the direct sum of two uniform matroids, $\U_{k-r,E\smallsetminus F}\oplus \U_{r,F}$. 
\end{theorem}

\begin{remark}\label{rem:notation-cuspidal}
    For the particular case in which $F = \{1,\ldots,h\}$, where $h = |F|$, it is customary to write $\LL_{r,k,h,n}$ and $\PP_{r,k,h,n}$, which closely follows the notation of \cite[Section~3]{ferroni-schroter}. (In that article, only isomorphism classes of these matroids are relevant so the choice of the labelling of $F$ is not taken into account.)
\end{remark}

\subsubsection{Formula for corank 2 matroids.}\label{subsec:corank2}

 We now describe the equivariant KL polynomial of arbitrary matroids of corank $2$.\footnote{One may obtain similar formulas for the equivariant $Z$-polynomials of arbitrary matroids of corank $2$ if one has closed formulas for uniform matroids of corank $1$ and $2$.} The key observation is the following: all matroids of corank $2$ can be written as a direct sum of loops, coloops and one connected (elementary) split matroid. By basic properties of equivariant KL polynomials we may disregard the loops and the coloops and focus on the connected summand that remains. This strategy has been used by Ferroni and Schr\"oter in \cite[Section~9.1]{ferroni-schroter} in order to compute the non-equivariant KL and $Z$-polynomials of corank $2$ matroids. To simplify the exposition we will in fact assume that our matroid is indeed connected. 

If we specialize Theorem~\ref{thm:equiv-val-for-split} for the case in which $k=n-2$, we need to describe in a reasonable way the matroids $\LL_{r,n-2,F,E}$ and $\PP_{r,n-2,F,E}$ appearing in that statement. In any connected matroid of corank $2$ all the restrictions at stressed flats are uniform of corank $1$, i.e., $\M|_F=\U_{|F|-1,F}$.  In particular, $\LL_{r,n-2,F,E} = \LL_{|F|-1,n-2,F,E}$, while $\PP_{r,n-2,F,E} = \PP_{|F|-1,n-2,F,E} \cong \U_{n-1-|F|,n-|F|}\oplus \U_{|F|-1,|F|}$, i.e., the last matroid is a direct sum of two corank $1$ uniform matroids. Hence, specializing Theorem~\ref{thm:equiv-val-for-split} for connected corank $2$ matroids, the statement reads as follows.

\begin{proposition}
    Let $W \curvearrowright \M$ be an equivariant connected matroid of rank $n-2$ on the ground set $E=[n]$ and let $\mathcal{F}$ denote its family of stressed flats. Then,
    \begin{multline*}    
    P_\M^W(x) = P_{\U_{n-2,n}}^W(x)\enspace - \\
    \sum_{F \in \mathcal{F}/W} \left(\Ind_{W_F}^W \left(P_{\LL_{|F|-1,n-2,F,E}}^{W_F}(x)\right) - \Ind_{W_F}^W \left( P_{\U_{n-1-|F|,n-|F|}}^{W_F}(x) \boxtimes P_{\U_{|F|-1,F}}^{W_F}(x)\right)\right).
    \end{multline*}    
\end{proposition}

In particular, all except one of the equivariant KL polynomials appearing on the right-hand side follow immediately by the discussion in Section~\ref{subsec:glued-cycles} and Theorem~\ref{thm:unif-equiv-sn}, as they correspond to uniform matroids of corank $1$ and $2$. Thus, the only remaining difficult computation that one needs to perform is $P_{\LL_{|F|-1,n-2,F,E}}^{W_F}(x)$.

The main strategy to simplify this calculation is to leverage that the base polytope of $\LL_{|F|-1,n-2,F,E}$ admits an explicit matroid subdivision. The crucial observation made in \cite[Proposition~5.6]{ferroni-schroter} is that all the maximal pieces in such subdivision are matroids of the form $\C_{a,b}$. Therefore, we now categorify \cite[Corollary~5.7]{ferroni-schroter}.

\begin{lemma}
    Let $W$ be a group acting on the matroid subdivision described in \cite[Proposition~5.6]{ferroni-schroter} of the matroid $\LL_{r,n-2,r+1,n}$ (cf. Remark~\ref{rem:notation-cuspidal}). Then,
    \begin{multline*} 
    P_{\LL_{r,n-2,r+1,n}}^W(x) = \\ \sum_{a=2}^{n-r-1} P^W_{\C_{a,n+1-a}}(x) - \sum_{a=2}^{n-r-2} \Ind_{(\mathfrak{S}_{a}\times \mathfrak{S}_{n-a}) \cap W}^{W} \Res_{(\mathfrak{S}_{a}\times \mathfrak{S}_{n-a}) \cap W}^{\mathfrak{S}_{a}\times \mathfrak{S}_{n-a}} \left( P^{\mathfrak{S}_{a}}_{\C_{a}}(x)\boxtimes P^{\mathfrak{S}_{n-a}}_{\C_{n-a}}(x)\right),
    \end{multline*}
    where $\mathfrak{S}_{a}\times \mathfrak{S}_{n-a}$ denotes the product of the actions that permute respectively the first $a$ and the last $n-a$ elements of $E=[n]$.
\end{lemma}

\begin{proof}
    This follows from an immediate application of \cite[Proposition~9.6]{elias-miyata-proudfoot-vecchi} because $W$ is assumed to act on the matroid subdivision.
\end{proof}

\begin{remark}
    We note that the group $\mathfrak{S}_{E\smallsetminus F}$ always acts on the subdivision mentioned above, because all the matroids that appear on the subdivision have $E\smallsetminus F$ as an independent set.
\end{remark}

Since we know how to compute the equivariant KL polynomial of the matroids $\C_{A,B}$, the preceding lemma implies that we are able to compute in a satisfactory way the equivariant KL polynomial of $\LL_{|F|-1,n-2,F,E}$, and thus of any connected corank $2$ matroid on an arbitrary action $W$ that preserves the matroid subdivision described in \cite[Proposition~5.6]{ferroni-schroter}.

\section*{Acknowledgments}

We thank the organizers of the Oberwolfach workshop ``Arrangements, matroids, and logarithmic vector fields''. A substantial part of this work was carried out at MFO. We also thank Tom Braden, Chris Eur, Nicholas Proudfoot, and Matthew Stevens for various conversations that helped us during the writing process of this article.  We thank two anonymous referees for suggestions and edits that greatly improved the quality of the paper.

\bibliographystyle{amsalpha}
\bibliography{bibliography}

\end{document}